\documentclass[12pt]{amsart}
\usepackage[active]{srcltx}
\usepackage{latexsym}
\usepackage{hyperref}
\usepackage{url}
\usepackage{amsmath}
\newcommand{\ignore}[1]{}
\newcommand{\hide}[1]{}
\DeclareMathOperator{\E}{E}

\DeclareMathOperator{\AH}{AH}

\newcommand{\F}{\mathbb F}

\newcommand{\Z}[0]{\mathbb Z}

\newcommand{\Q}{\mathbb{Q}}

\newtheorem{dummy}{Dummy}

\newtheorem{lemma}[dummy]{Lemma}

\newtheorem{theorem}[dummy]{Theorem}

\newtheorem{prop}[dummy]{Proposition}

\theoremstyle{definition}

\theoremstyle{remark}
\newtheorem{rem}[dummy]{Remark}
\newtheorem*{rem*}{Remark to ourselves}
\begin{document}%_______________________________________

\bibliographystyle{amsalpha}

\author{Marina Avitabile}
\email{marina.avitabile@unimib.it}
\address{Dipartimento di Matematica e Applicazioni\\
  Universit\`a degli Studi di Milano - Bicocca\\
 via Cozzi 55\\
  I-20125 Milano\\
  Italy}
\author{Sandro Mattarei}
\email{smattarei@lincoln.ac.uk}
\address{Charlotte Scott Centre for Algebra\\
University of Lincoln \\
Brayford Pool
Lincoln, LN6 7TS\\
United Kingdom}

\title{The Artin-Hasse series and Laguerre polynomials modulo a prime}

\subjclass[2020]{Primary 33E50; secondary 33C45}
%    The 2010 edition of the Mathematics Subject Classification is
%    now available.  If you are citing a classification from the
%    new scheme, use the following input coding instead.
\keywords{Artin-Hasse series, Laguerre polynomials}

\begin{abstract}
For an odd prime $p$, let
$\E_{p}(X)=\sum_{n=0}^{\infty} a_{n}X^{n}\in\F_p[[X]]$
denote the reduction modulo $p$ of the Artin-Hasse exponential series.
It is known that there exists a series $G(X^p)\in \F_{p}[[X]]$,
such that $L_{p-1}^{(-T(X))}(X)=\E_{p}(X)\cdot G(X^p)$, where
$T(X)=\sum_{i=1}^{\infty}X^{p^{i}}$ and $L_{p-1}^{(\alpha)}(X)$ denotes the
(generalized) Laguerre polynomial of degree $p-1$.
We prove that $G(X^p)=\sum_{n=0}^{\infty}(-1)^n a_{np}X^{np}$, and show that it satisfies
$
G(X^p)\,G(-X^p)\,T(X)=X^p.
$

%\vspace{1cm}
%Version \today
\end{abstract}

\date{\today}

\maketitle

\section{Introduction}\label{sec:intro}
Let $p$ be a prime. The Artin-Hasse exponential
series is the formal power series in $\Q[[X]]$ defined as
\begin{equation*}
\AH(X)=\exp\biggl(\sum_{i=0}^{\infty} X^{p^i}/p^{i}\biggr)
=\prod_{i=0}^{\infty}\exp \left(X^{p^i}/p^{i}\right).
\end{equation*}

As an immediate application of the Dieudonn\'e-Dwork criterion, its coefficients
are $p$-integral, hence they can be evaluated modulo $p$.
Let $\E_{p}(X)=\sum a_{n}X^{n}$ denote the reduction modulo $p$
of the Artin-Hasse exponential series,
hence viewed as a series in $\F_p[[X]]$.

The series $\E_{p}(X)$  satisfies a weak version of the
functional equation $\exp(X)\exp(Y)=\exp(X+Y)$ of the classical exponential
series $\exp(X)=\sum X^k/k!$ in characteristic zero.
In fact, it was shown in~\cite[Theorem 2.2]{Mat:Artin-Hasse} that each term of the series
$(\E_{p}(X+Y))^{-1}\,\E_{p}(X)\,\E_{p}(Y)$ has degree a multiple of $p$.
This weak functional equation satisfied by $\E_{p}(X)$ is the crucial property needed
for a {\em grading switching} technique developed in~\cite{Mat:Artin-Hasse}, whose goal is producing a new grading of a non-associative algebra $A$ in characteristic $p$ from a given grading.
Roughly speaking, the new grading of $A$ is obtained by applying $\E_{p}(D)$ to each homogeneous components of the given grading, where $D$ is a nilpotent derivation of $A$ satisfying a certain compatibility condition with the grading.
In full generality, that is for arbitrary derivations, the grading switching was developed in~\cite{AviMat:Laguerre}.
The {\em toral switching}  (see~\cite{Win:toral}, \cite{BlWil:rank-two} and~\cite{Premet:Cartan}), a fundamental tool in the classification theory of simple modular Lie algebras, can be recovered as a special case of it.

In this Introduction we limit ourselves to a brief survey of the definitions and the results which are essential for the purposes of this paper, referring the reader to Section~\ref{sec:prel} for further details, and the interested reader to~\cite{AviMat:Laguerre} and~\cite{AviMat:gradings} for full details.
In the general case of the grading switching, the role of the Artin-Hasse exponential is played by the (generalized) Laguerre polynomials of degree $p-1$,
\begin{equation}\label{eq:Laguerre}
L_{p-1}^{(\alpha)}(X)=\sum_{k=0}^{p-1}\binom{\alpha-1}{p-1-k}\frac{(-X)^k}{k!},
\end{equation}
regarded as polynomials in $\F_{p}[\alpha,X]$. These polynomials satisfy a congruence  which can be interpreted as a further generalization of the weak functional equation satisfied by $\E_{p}(X)$.
The main result of~\cite{Mat:exponential} then implies that the power series
\[
 S(X)=L_{p-1}^{(-\sum_{i=1}^{\infty}X^{p^{i}})}(X),
\]
in $1+X\F_{p}[[X]]$,
satisfies
$S(X)=E_{p}(X)\cdot G(X^p)$ for some series $G(X)\in 1+X\F_{p}[[X]]$.
Our main result is to determine the coefficients of the series $G(X)$ in terms of those of $\E_{p}(X)$.
We separately deal with the case $p=2$ in Remark~\ref{rem:p=2}. When $p$ is an odd prime we have the following.

\begin{theorem}\label{thm:G}
Let $p$ be an odd prime,
and let $\E_{p}(X)=\sum_{n=0}^{\infty} a_{n}X^{n}$ in $\F_{p}[[X]]$
be the reduction modulo $p$ of the Artin-Hasse exponential series. Then
\begin{equation}\label{eq:G}
L_{p-1}^{(-\sum_{i=1}^{\infty}X^{p^{i}})}(X)=\E_{p}(X)\cdot \sum_{n=0}^{\infty}(-1)^n a_{np}X^{np},
\end{equation}
and
\begin{equation}\label{eq:F}
L_{p-1}^{(-\sum_{i=1}^{\infty}X^{p^{i}})}(X)
\cdot\biggl(\sum_{i=1}^{\infty}X^{p^{i}}\biggr)
\cdot \sum_{n=0}^{\infty} a_{np}X^{np}=X^p\E_{p}(X)
\end{equation}
in $\F_{p}[[X]]$.
\end{theorem}

As an immediate consequence of Theorem~\ref{thm:G} we have the following result.

\begin{prop}\label{prop:X^p}
Let $p$ be an odd prime and $\E_{p}(X)=\sum_{i=0}^{\infty} a_{i}X^{i}$ in $\F_{p}[[X]]$ the reduction modulo $p$ of the Artin-Hasse exponential series.
Then in $\F_{p}[[X]]$ we have the identity
\begin{equation}\label{eq:+-}
\sum_{s=0}^{\infty} a_{sp}X^{sp}
\cdot\sum_{r=0}^{\infty} a_{rp}(-X)^{rp}
\cdot\sum_{i=1}^{\infty}X^{p^i}=X^p.
\end{equation}
\end{prop}

Proposition~\ref{prop:X^p} can equivalently be phrased in the following form
involving the series $\sum_{p\mid k}X^k/k!$ in place of the Artin-Hasse series.

\begin{prop}\label{prop:e_p}
For any odd prime $p$, in $\Q[[X]]$ we have
\[
\sum_{s=0}^{\infty}\frac{X^{sp}}{(sp)!}
\cdot\sum_{r=0}^{\infty}\frac{(-X)^{rp}}{(rp)!}
\cdot\sum_{i=1}^{\infty}X^{p^i}\equiv X^p\pmod{p}.
\]
\end{prop}

Despite its appearance, the left-hand side of the congruence of Proposition~\ref{prop:e_p}
has $p$-integral coefficients, which justifies viewing it modulo $p$.
This variant does not seem to offer any more direct proof than
deducing it from Proposition~\ref{prop:X^p}, which we will do
in Section~\ref{sec:proofs}.

As a final application of Theorem~\ref{thm:G}, we use properties of the Laguerre polynomials
to produce explicit expressions for the coefficients $a_n$, with $0\le n<p^2$, in terms
of coefficients in the same range but with $n$ multiple of $p$.
We denote by ${n \brack i}$ the (unsigned) Stirling numbers of the first kind.

\begin{prop}\label{le:a_rp+k}
For $0\leq k<p$ and $0\leq r<p$ we have
\[
a_{rp+k}=(-1)^{k+1}\sum_{j=0}^{r}{p-k \brack j+1}c_{(r-j)p},
\]
where $c_{jp}=a_{jp}$ for $0\leq j <p-1$ and $c_{(p-1)p}=a_{(p-1)p}+1$.
%for $0\leq i \leq n$
\end{prop}

We present proofs of our results in Section~\ref{sec:proofs}.

\section{Preliminaries}\label{sec:prel}%%%%%%%%%%%%%%%%%%%%%%%%%%%%%%%%%%%%
%%%%%%%%%%%%%%%%%%%%%%%%%%%%%%%%%%%%%%%%%%%%%%%%%%%%%%%%%%%%%%%%%%%%%%%%

Let $\Z_{p}$ denote the ring of $p$-adic integers, where $p$ is any prime, and write
\begin{equation}\label{eq:AH_def}
\AH(x)=\sum_{n=0}^{\infty}u_{n}X^{n} \in \Z_{p}[[X]].
\end{equation}

The coefficients $u_{n}$ satisfy the {\em recursive formula}
(see \cite[Lemma~1]{KS})
\begin{equation}\label{eq:FR}
u_{n}=\frac{1}{n}\sum_{i=0}^{\infty}u_{n-p^i},
\end{equation}
where $u_{0}=1$ and we naturally read $u_{m}=0$ for $m<0$, which easily follows
from differentiating Equation~\eqref{eq:AH_def}.

Our interest lies exclusively in prime characteristic $p$.
Denote by
$\E_{p}(X)=\sum a_{n}X^n \in\F_{p}[[X]]$
the reduction modulo $p$ of the Artin-Hasse exponential series, hence $a_n \equiv u_n$ modulo $p$.
As mentioned in the Introduction,  $\E_{p}(X)$ satisfies a functional equation which is a weak version of the fundamental equation $\exp(X)\exp(Y)=\exp(X+Y)$ for the classical exponential
series $\exp(X)=\sum X^k/k!$ in characteristic zero. Namely, as shown in the proof
of~\cite[Theorem 2.2]{Mat:Artin-Hasse}, we have
\begin{equation}\label{eq:Ep_fe}
\E_{p}(X)\E_{p}(Y)=\E_{p}(X+Y)\left(1+\sum_{i,j}a_{i,j}X^iY^j\right)
\end{equation}
in $\F_{p}[[X,Y]$, for some coefficients $a_{i,j}\in \F_{p}$ which vanish unless $p\;\vert \;i+j$.
The functional equation~\eqref{eq:Ep_fe} actually characterizes the series $\E_{p}(X)$ in $\F_{p}[[X]]$,
up to some natural variations. Precisely, we quote
from~\cite{Mat:exponential} the following
\begin{theorem}[{\cite{Mat:exponential}}]\label{thm:mat_exp}
For a series $S(X)\in 1+X\F_{p}[[X]]$, the series
\[
(S(X+Y))^{-1}\,S(X)\,S(Y)\in \F_{p}[[X,Y]]
\]
has only terms of total degree a multiple of $p$ if and only if
\[
S(X)=\E_{p}(cX)\cdot G(X^p),
\]
for some $c\in \F_{p}$ and $G(X)\in 1+X\F_{p}[[X]]$.
\end{theorem}

The classical (generalized) Laguerre polynomial of degree $n \geq 0$ is defined as
\[
L_{n}^{(\alpha)}(X)=\sum_{k=0}^{n}\binom{\alpha+n}{n-k}\frac{(-X)^k}{k!},
\]
where $\alpha$ is a parameter, usually in the complex field. However, we may also view
$L_{n}^{(\alpha)}(X)$ as a polynomial with rational coefficients in the two indeterminates
$\alpha$ and $X$, hence in the polynomial ring $\Q[\alpha,X]$.
We are only interested in the Laguerre polynomials of degree $n=p-1$.
Their coefficients are $p$-integral, and hence can be evaluated modulo $p$.
In particular, $L_{p-1}^{(\alpha)}(X)$
will be viewed as a polynomial in $\F_p[\alpha,x]$, and as such will be given by Equation~\eqref{eq:Laguerre}.
Note that, for $\alpha=0$, $L_{p-1}^{(0)}(X)$ equals
the {\em truncated exponential} $\E(X)=\sum_{k=0}^{p-1}X^k/k!$, which in turns is congruent to
$\E_{p}(X)$ modulo $X^p$.

The Laguerre polynomials $L_{p-1}^{(\alpha)}(X)$ satisfy a congruence which can be interpreted as a further generalization of Equation~\eqref{eq:Ep_fe}. Indeed, it follows from~\cite[Proposition 2]{AviMat:Laguerre} (but see also~\cite[Theorem 1]{AviMat:gfpolylogs} for a streamlined statement) that there exist rational expressions $c_{i}(\alpha,\beta)\in \F_{p}(\alpha,\beta)$ such that
\begin{equation}\label{eq:Laguerre_fe}
L_{p-1}^{(\alpha)}(X)L_{p-1}^{(\beta)}(Y)\equiv L_{p-1}^{(\alpha+\beta)}(X+Y)\left(c_{0}(\alpha,\beta)+
\sum_{i=1}^{p-1}c_{i}(\alpha,\beta)X^iY^{p-i}\right),
\end{equation}
in $\F_{p}(\alpha,\beta)[X,Y]$, modulo the ideal generated by $X^p-(\alpha^p-\alpha)$ and
$Y^p-(\beta^p-\beta)$.
This congruence actually characterizes the polynomials $L_{p-1}^{(\alpha)}(X)$ among those in
$\F_{p}[\alpha][X]$, up to some natural variations, as proved in~\cite[Theorem 3]{AviMat:gfpolylogs}.

In the rest of the paper we let $S(X)$ denote the power series in $1+X\F_{p}[[X]]$ defined as
\[
 S(X)=L_{p-1}^{(-\sum_{i=1}^{\infty}X^{p^{i}})}(X).
\]
According to~\cite[Proposition 6]{AviMat:gradings} to which we refer for details, Equation~\eqref{eq:Laguerre_fe} implies that $(S(X+Y))^{-1}S(X)S(Y)$ has only terms of
degree divisible by $p$. According to Theorem~\ref{thm:mat_exp}, since $S(X)\equiv
L_{p-1}^{(0)}(X)=\E(X)\equiv \E_{p}(X)$ modulo $X^p$,
we have
\begin{equation}\label{eq:S=EG}
S(X)=\E_{p}(X)\cdot G(X^{p})
\end{equation}
for some $G(X)$ in $1+X\F_{p}[[X]]$. Equivalently, we have
\begin{equation}\label{eq:SF=E}
S(X)\cdot F(X^p)=\E_{p}(X),
\end{equation}
for some $F(X)=1/G(X)$ in $1+X\F_{p}[[X]]$.
Our Theorem~\ref{thm:G} produces explicit expressions for $G(X^p)$ and $F(X^p)$.

\section{Proofs}\label{sec:proofs}%%%%%%%%%%%%%%%%%%%%%%%%%%%%%%%%%%%%%%%%%%%%%%%%%%

In this section we prove Theorem~\ref{thm:G}, Proposition~\ref{prop:X^p} and Proposition~\ref{le:a_rp+k}.
We will need the following special instance of  Equation~\eqref{eq:Laguerre_fe}.

\begin{lemma}[{\cite[Lemma 10]{AviMat:glog}}]\label{le:S(x)S(-X)}
In the polynomial ring $\F_{p}[\alpha,X]$ we have
\[
L_{p-1}^{(\alpha)}(X)\cdot L_{p-1}^{(-\alpha)}(-X)\equiv 1-\alpha^{p-1} \pmod {X^p-(\alpha^p-\alpha)}.
\]
\end{lemma}

Note that $L_{p-1}^{(\alpha)}(0)=\binom{\alpha-1}{p-1}=1-\alpha^{p-1}=L_{p-1}^{(-\alpha)}(0)$.

\begin{proof}[Proof of Theorem~\ref{thm:G}]
From Equation~\eqref{eq:S=EG}
and the fact that
$\E_{p}(X)\E_{p}(-X)=1$ (for $p$ odd)
we deduce
\[
T(X)\cdot S(X)\cdot S(-X)\cdot \E_{p}(-X)
=T(X)\cdot S(-X)\cdot G(X^p),
\]
where $T(X)=\sum_{i=1}^{\infty}X^{p^i}$.
To fix notation we set $G(X^p)=\sum_{n=0}^{\infty}b_{np}X^{np}$.
Setting $\alpha=-T(X)$
in Lemma~\ref{le:S(x)S(-X)} we find
$
S(X)\cdot S(-X)
=1-T(X)^{p-1}
$.
In more formal terms
we have applied to the congruence the ring homomorphism
of $\F_{p}[\alpha,X]$ to $\F_p[[X]]$ which maps $\alpha$ to $-T(X)$,
noting that the modulus
$X-(\alpha^p-\alpha)$
belongs to its kernel.
Consequently,
$T(X)\cdot S(X)\cdot S(-X)=X^p$,
because $T(X)-T(X)^p=X^p$, and hence
\[
X^p\cdot
\sum_{n=0}^{\infty}(-1)^n a_{n}X^{n}
=T(X)\cdot S(-X)\cdot
\sum_{n=0}^{\infty}b_{np}X^{np}.
\]
Now we are only interested in the terms of this equation where the exponent of $X$ is a multiple of $p$.
In the case of $S(-X)=L_{p-1}^{(T(X))}(-X)$ the collection of such terms equals
\[
\binom{T(X)-1}{p-1}=1-T(X)^{p-1}
\]
in $\F_p(X)$.
Because $T(X)-T(X)^p= X^p$ we conclude
\[
X^p \sum_{n=0}^{\infty}(-1)^n a_{np}X^{np}= X^p \sum_{n=0}^{\infty} b_{np}X^{np},
\]
which is equivalent to Equation~\eqref{eq:G}.

To prove Equation~\eqref{eq:F} we proceed in a similar way, starting from
$S(X)\cdot F(X^p)=\E_{p}(X)$ in $\F_{p}[[X]]$.
Setting $F(X^p)=\sum_{n=0}^{\infty}c_{np}X^{np}$
we have
\[
T(X)\cdot S(X)\cdot \sum_{n=0}^{\infty}c_{np}X^{np} =
T(X)\cdot \sum_{n=0}^{\infty}a_{n}X^{n}.
\]
Restricting to powers of $X$ with exponent a multiple of $p$ in each side we obtain
\[
X^p \cdot \sum_{n=0}^{\infty}c_{np}X^{np} = T(X)\cdot \sum_{n=0}^{\infty}a_{np}X^{np},
\]
which is equivalent to Equation~\eqref{eq:F}.
\end{proof}

\begin{rem}\label{rem:p=2}
Although Theorem~\ref{thm:G} does not extend to $p=2$ as stated, a replacement for Equation~\eqref{eq:F}
is easily found directly.
Indeed, $L_{1}^{(\alpha)}(X)=1+\alpha+X$ and the recursive formula Equation~\eqref{eq:FR} implies
$
a_{2n}=a_{2n+1}+\sum_{i=1}^{\infty}a_{2n+1-2^{i}}
$
for every integer $n$, where $a_{n}=0$ for $n<0$. Hence,
\begin{align*}
\E_{2}(X)&=\sum_{n=0}^{\infty}a_{2n}X^{2n}+\sum_{n=0}^{\infty}a_{2n+1}X^{2n+1}=\\
&=(1+X)\sum_{n=0}^{\infty}a_{2n+1}X^{2n}+\sum_{n=0}^{\infty}\left(\sum_{i=1}^{\infty}a_{2n+1-2^{i}}\right)X^{2n}=\\
&=(1+X+\sum_{i=1}^{\infty}X^{2^{i}})\sum_{n=0}^{\infty}a_{2n+1}X^{2n}=S(X)\sum_{n=0}^{\infty}a_{2n+1}X^{2n}.
\end{align*}
Thus, when $p=2$ Equation~\eqref{eq:SF=E} holds with $F(X^2)=\sum_{n=0}^{\infty}a_{2n+1}X^{2n}$.
\end{rem}

Theorem~\ref{thm:G} immediately implies Proposition~\ref{prop:X^p}.
\begin{proof}[Proof of Proposition~\ref{prop:X^p}]
Our goal can be restated as
\[
G(X^p)\,G(-X^p)\,T(X)=X^p.
\]
Now $G(X^p)\,G(-X^p)=S(X)\,S(-X)=1-T(X)^{p-1}$,
as we deduced from Lemma~\ref{le:S(x)S(-X)} at the beginning of the proof of Theorem~\ref{thm:G}.
The conclusion follows because $G(X^p)=\sum_{r=0}^{\infty}(-1)^ra_{rp}X^{rp}$ according to Theorem~\ref{thm:G}.
\end{proof}

Deducing Proposition~\ref{prop:e_p} from Proposition~\ref{prop:X^p}
requires the technique of series multisection.

\begin{proof}[Proof of Proposition~\ref{prop:e_p}]
In terms of
$e_p(X)
=\sum_{p\mid k}X^k/k!$,
our goal becomes the congruence
\[
e_p(X)\,e_p(-X)\,\sum_{i=1}^{\infty}X^{p^i}
\equiv X^p\pmod{p}
\]
from the equation
\[
G(X^p)\,G(-X^p)\,T(X)=X^p
\]
in $\F_p[[X]]$.
We have
$
e_p(X)
=(1/p)\sum_{\omega^p=1}\exp(\omega X),
$
where the sum is over all complex $p$th roots of unity $\omega$.

Because of the equation
\[
\AH(X)
=\prod_{i=0}^{\infty}\exp \left(X^{p^i}/p^{i}\right)
=\exp(X)\AH(X^p)^{1/p},
\]
our series $G(-X^p)$
equals the reduction modulo $p$ of
$e_p(X)\AH(X^p)^{1/p}$.
Consequently, for $p$ odd, the product
$G(X^p)\,G(-X^p)$ equals the reduction modulo $p$ of
\[
e_p(X)\AH(X^p)^{1/p}\cdot e_p(-X)\AH(-X^p)^{1/p},
\]
which simplifies to
$e_p(X)\,e_p(-X)$.
Note that $e_p(X)\AH(X^p)^{1/p}$
belongs to $\Z_p[[X]]$
because
$\AH(X)$ does.
Hence so does $e_p(X)\,e_p(-X)$.
\end{proof}

Denote by $y^{\overline{n}}=y(y+1)\cdots (y+n-1)$ the {\em rising factorial}.
The (unsigned) Stirling number of the first kind ${n \brack i}$,
for $0\leq i \leq n$, may be defined by the polynomial identity $y^{\overline{n}}=\sum_{i=0}^{n} {n \brack i} y^i$
in $\Z[y]$.

\begin{proof}[Proof of Proposition~\ref{le:a_rp+k}]
In view of working modulo $X^{p^2}$,
because
$S(X)$ is congruent with $L_{p-1}^{(-X^{p})}(X)$,
we expand the latter as
\begin{align*}
&L_{p-1}^{(-X^{p})}(X)=\sum_{k=0}^{p-1}\binom{X^p-(k+1)}{p-1-k}\frac{X^k}{k!}=
\sum_{k=0}^{p-1}(-1)^{k+1}(X^p+1)^{\overline{p-1-k}}X^k \\
&=\sum_{k=0}^{p-1}\sum_{i=0}^{p-1-k}{p-1-k \brack i} (X^p+1)^{i}(-1)^{k+1}X^k\\
&=\sum_{k=0}^{p-1}\sum_{j=0}^{\infty}(-1)^{k+1}\left(\sum_{i=j}^{p-1-k}{p-1-k \brack i}\binom{i}{j}\right) X^{pj+k}
\\&=
\sum_{k=0}^{p-1}\sum_{j=0}^{\infty}(-1)^{k+1}{p-k \brack j+1} X^{pj+k},
\end{align*}
where we have used the standard identity $\sum_{t=m}^{n}{n \brack t}\binom{t}{m}={n+1 \brack m+1}$.

Because of Equation~\eqref{eq:SF=E} we have
$L_{p-1}^{(-X^{p})}(X)F(X^p)\equiv \E_{p}(X)\pmod{X^{p^2}}$, where $F(X^p)=\sum_{r=0}^{\infty}c_{rp}X^{rp}$ for some $c_{np}\in \F_{p}$.
Comparing this with
\[
L_{p-1}^{(-X^{p})}(X)F(X^p)\equiv \sum_{k=0}^{p-1}\sum_{r=0}^{p-1}\sum_{j=0}^{r}(-1)^{k+1}{p-k \brack j+1}c_{(r-j)p}X^{rp+k}
\pmod{X^{p^{2}}}
\]
completes the proof.
Note that the equation in the statement implicitly includes a definition of the
coefficients $c_{jp}$ when $k=0$.
\end{proof}

\begin{rem}
The coefficients $u_n\in\Q$ of the Artin-Hasse series may be computed recursively
from Equation~\eqref{eq:FR}.
When $n$ is not a multiple of $p$, the recursive equation may be read modulo $p$, and hence applied directly to the coefficients $a_n$.
Writing $n=rp+k$, with $0\le k<p$, a recursive application of Equation~\eqref{eq:FR} shows that
$a_{rp+k}$ may eventually be computed from the coefficients $a_{ip}$ for $i<r$.
Proposition~\ref{le:a_rp+k} provides an explicit form for the final result of that process.
\end{rem}

\bibliography{References}

\end{document}